\documentclass[10pt]{article}

\usepackage[tmargin=1.0in]{geometry}
\newfont{\smcaps}{cmcsc10 scaled\magstep1}

\newcommand{\VARMA}{{\rm VARMA\,}}
\newcommand{\AR}{{\rm AR\,}}
\newcommand{\VAR}{{\rm VAR\,}}

\newcommand{\VEC}{{\rm ~vec\,}}

\newcommand{\diag}{{\rm ~diag\,}}

\newcommand{\tr}{{\rm ~tr\,}}

\newcommand{\Cov}{{\rm ~Cov\,}}

\usepackage{latexsym}
\usepackage{amsthm}
\usepackage{bm} 
\usepackage{amsmath,amsfonts,amssymb}
\usepackage{natbib}
\usepackage{graphicx}
\usepackage{url}
\usepackage{colortbl}

\raggedright
\begin{document}

\title{IMPROVED MULTIVARIATE PORTMANTEAU TEST}
\date{\today}
\author{{Esam Mahdi and A. Ian McLeod}\\
{\it The University of Western Ontario\/}
}

\maketitle
\bigskip
\hrule
\bigskip
{\noindent
Corresponding Author:\\
A.I. McLeod\\
Department of Statistical and Actuarial Sciences,\\
The University of Western Ontario,\\
London, Ontario N6A 5B7\\
Canada\\
email: aimcleod@uwo.ca.
}

\newpage
\baselineskip=22pt
\hrule
\bigskip
{\bf Abstract.\/}
A new portmanteau diagnostic test for vector autoregressive
moving average (\VARMA) models that
is based on the determinant of the
standardized multivariate residual autocorrelations is derived.
The new test statistic may be considered an extension of the univariate portmanteau test statistic
suggested by \citet{PR2002}.
The asymptotic distribution of the test statistic is derived as well
as a chi-square approximation.
However, the Monte-Carlo test is recommended unless the series is very long.
Extensive simulation experiments demonstrate the usefulness
of this test as well as its improved power performance compared
to widely used previous multivariate portmanteau diagnostic check.
Two illustrative applications are given.

\VARMA models.

\newpage

\section{Introduction}\label{secIntroduction}

The \VARMA$(p,q)$ model for a $k$-dimensional mean zero time series
$\bm{Z}_{t}=(Z_{1,t},\ldots,Z_{k,t})^{{\prime}}$
can be written as
\begin{equation}\label{VARMA}
\bm{\Phi(B)Z}_{t}=\bm{\Theta(B) a}_{t},
\end{equation}
where
$\bm{\Phi(B)}=\mathbb{I}_{k}-\bm{\Phi}_{1}\bm{B}-\cdots-\bm{\Phi}_{p}{\bm{B}}^{p},
\bm{\Theta(B)}=\mathbb{I}_{k}-\bm{\Theta}_{1}\bm{B}-\cdots-\bm{\Theta}_{q}{\bm{B}}^{q}$,
$\mathbb{I}_{k}$ is the identity matrix of order $k$,
the coefficient matrices are,
$\bm{\Phi}_\mathit{\ell}=(\phi_{i,j,\ell})_{k\times k}, \ell=1,\ldots,p$;
$\bm{\Theta}_\ell=(\theta_{i,j,\ell})_{k\times k}, \ell=1,\ldots,q$
and $\bm{B}$ is the backshift operator on $t$.
Let
$\bm{\beta} = ( \text{vec}\ \bm{\Phi}_1, \ldots, \text{vec}\ \bm{\Phi}_p, \text{vec}\ \bm{\Theta}_1,\ldots, \text{vec}\ \bm{\Theta}_q)$
be the vector of true parameters, where \text{vec} denotes the matrix vectorization function.
We assume that an efficient estimation
algorithm such as maximum likelihood is used to produce the corresponding estimate
$\hat{\bm{\beta}}$
so that $\hat{\bm{\beta}}-\bm{\beta} = O_p(n^{-1/2})$.
The white noise process,
$\bm{a}_{t}=(a_{1,t},\ldots,a_{k,t})^{{\prime}}$,
is assumed independent normal with mean zero and covariance matrix,
$E(\bm{a}_{t}\bm{a}_{t-\ell}^{\prime}) = \delta_\ell \bm{\Gamma}_{0}$,
where
$\bm{\Gamma}_{0}$
is the innovation covariance matrix and
$\delta_\ell = 1$ or $0$ according as $\ell = 0$ or $\ell \ne 0$.
The assumption of normality may be relaxed to that of strong white noise so
that $\bm{a}_{t}, t=1,\ldots,n$ are assumed to be independent and identically
distributed with mean zero and constant covariance matrix, $\bm{\Gamma}_{0}$.
The model is assumed to be stationary,
invertible, and identifiable \citep[\S 14.2]{Box2008}.
After fitting this model to a series of length $n$, the residuals, $\bm{\Hat{a}}_{t}=(\Hat{a}_{1,t},\ldots,\Hat{a}_{k,t})^{{\prime}}$,
$t = 1,\ldots,n$ may be estimated
and used to check the model assumption that the innovations are white noise, that is,
to test the null hypothesis that
\begin{equation}\label{nullHypothesis}
{\cal H}_0: \bm{\Gamma}_{\ell} = 0,\ \ell=1,\ldots,m,
\end{equation}
where $\bm{\Gamma}_{\ell} = \Cov\{\bm{a}_{t}, \bm{a}_{t-\ell}\}$
and $m$ is chosen large enough to cover all lags, $\ell$, of interest.
Several versions of the multivariate portmanteau test have been developed for this purpose \citep{Li2004}.

In the next two subsections, brief reviews are given of previous multivariate portmanteau tests
as well as the univariate versions of the generalized variance test of \citet{PR2002, PR2006}.
In Section~\ref{secNew}, the multivariate extension of the generalized variance test
of \citet{PR2002} is discussed and its asymptotic distribution is derived.
As in the univariate case \citep[eqn. (9)]{PR2002},
it is shown in eqn. (\ref{NewDecompositionPartB})
that the stronger the multivariate autocorrelation,
the smaller the generalized variance.
A chi-square approximation is suggested but
for most purposes it is recommended to use
a Monte-Carlo testing procedure that is described in Section \ref{subMonte}.
Simulation experiments in Section~\ref{secSimulation},
demonstrate the improvement in power over the widely used previous multivariate portmanteau test.
Illustrative applications are discussed is Section~\ref{secApplication}.

\subsection{Multivariate portmanteau tests}\label{Multivariate}

The portmanteau test statistics, $Q_m$ and $\Tilde{Q}_m$ and others, discussed
in this section are all asymptotically $\chi^2_{k^2(m-p-q)}$
as $n \rightarrow \infty$. It is also assumed that $m>p+q$ is fixed and that
$m$ large enough so that Theorem 5 in \citet{Li1981} holds.

\citet{Hosking1980B} defined the residual autocorrelation matrix,
\begin{equation}
\label{HostingRACF}
\Hat{\bm{R}}_{\mathit{\ell}} = \bm{\hat{L}}^{\prime}\Hat{\bm{\Gamma}}_{\mathit{\ell}}\bm{\hat{L}},
\end{equation}
where
$\Hat{\bm{\Gamma}}_\ell
=n^{-1}\sum_{t=\ell+1}^{n}\Hat{\bm{a}}_{t}{\Hat{\bm{a}}_{t-\ell}}^{\prime},
\Hat{\bm{\Gamma}}_{-\ell} = \Hat{\bm{\Gamma}}_{\ell}^{\prime},
\ell \ge 0
$
and $\bm{\hat{L}}$ is the lower triangular Cholesky decomposition of ${\Hat{\bm{\Gamma}}_{0}}^{-1}$.
The multivariate portmanteau test statistic may be written,
\begin{equation}\label{HoskingQ}
Q_m = n\sum_{\ell=1}^{m}{\Hat{\bm{r}}_\ell}^{{\prime}}({\Hat{\bm{R}}_{0}}^{-1}\otimes{\Hat{\bm{R}}_{0}}^{-1})\Hat{\bm{r}}_\ell,
\end{equation}
where
$\Hat{\bm{r}}_\ell=\text{vec}\ {\Hat{\bm{R}}_\ell}^{{\prime}}$
is a row vector of length $k^2$ formed by stacking the rows of $\Hat{\bm{R}}_\ell$,
and $m$ represents the number of lags being tested.
In the univariate case,  $Q_m$ is identical to Box-Pierce portmanteau statistic
\citep{Box1970}
and both statistics are asymptotically $\chi^2_{k^2(m-p-q)}$ \citep{Hosking1980B,Hosking1981B}.

\citet{Li1981} defined,
\begin{equation}\label{LiMcLeodRACF}
\Hat{\bm{R}}^{(\dag)}_{\ell}=(\Hat{r}_{i,j}(\ell))_{k\times k},
\end{equation}
where
$\Hat{r}_{i,j}(\ell) = \Hat{\gamma}_{i,j}(\ell)/\surd{(\Hat{\gamma}_{i,i}(0)\Hat{\gamma}_{j,j}(0))}$,
$i,j = 1,\ldots, k$ and
$\Hat{\gamma}_{i,j}(\ell)=n^{-1}\sum_{t=\ell+1}^{n}\Hat{a}_{i,t}\Hat{a}_{j,t-\ell}$,
$\Hat{\gamma}_{i,j}(-\ell) = \Hat{\gamma}_{j,i}(\ell), \ell \ge 0$.
Replacing $\Hat{\bm{R}}$ by $\Hat{\bm{R}}^{(\dag)}$ in eqn. (\ref{HoskingQ}), another
portmanteau test statistic $Q^{(\dag)}_m$ is obtained.
The null distribution of $Q^{(\dag)}_m$ is also asymptotically $\chi^2_{k^2(m-p-q)}$.
The definition of residual autocorrelations used in eqn. (\ref{HostingRACF})
is equivalent to the residual autocorrelations in eqn. (\ref{LiMcLeodRACF})
if the residuals used eqn. (\ref{LiMcLeodRACF}), $\Hat{\bm{a}}_{t}$, are replaced
by the standardized residuals, $\bm{\hat{L}}^{\prime} \Hat{\bm{a}}_{t}$.

\citet{Chitturi1974} defined the residual autocorrelation matrix at lag $\ell$,
\begin{equation}\label{ChitturiRACF}
\Hat{\bm{R}}^{(\ddag)}_\ell=\Hat{\bm{\Gamma}}_\ell{\Hat{\bm{\Gamma}}_{0}}^{-1},
\end{equation}
and another portmanteau test statistic $Q^{(\ddag)}_{m}$ is obtained by replacing
$\Hat{\bm{R}}$ by $\Hat{\bm{R}}^{(\ddag)}$ in eqn. (\ref{HoskingQ}), and
its null distribution is also asymptotically $\chi^2_{k^2(m-p-q)}$.

\citet{Hosking1981B} noted that
$Q_m = Q^{(\dag)}_{m} = Q^{(\ddag)}_{m}$ and the portmanteau test statistic may be expressed simply in terms of the
residual autocovariances,
\begin{equation}\label{generalMVportmanteauStat2}
Q_m=n\sum_{\ell=1}^{m} \tr({\Hat{\bm{\Gamma}}_{\mathit{\ell}}}^{{\prime}}{\Hat{\bm{\Gamma}}_{0}}^{-1}\Hat{\bm{\Gamma}}_{\mathit{\ell}}{\Hat{\bm{\Gamma}}_{0}}^{-1}),
\end{equation}
where \tr($\bullet$) denotes trace of matrix.
The multivariate portmanteau test statistic is equivalent to a test based on the
Lagrange multiplier \citep{Hosking1981A, Poskitt1982}.

\citet{Hosking1980B} and \citet{Li1981} suggested modified versions of $Q_m$
so that the expected value of the modified portmanteau statistic under the null hypothesis is
equal to $k^2(m-p-q)+O_p(1/n)$
and showed that both of these modifications are satisfactory when $n$ and $m$ are large enough.
Simulation experiments suggest that both these modified portmanteau tests work
about equally well \citep[\S 3]{Li2004}.

The modified portmanteau test of \citet{Hosking1980B} is given by,
\begin{equation}
\label{ModifiedMVTestHosking}
\Tilde{Q}_m
=
n^{2}\sum_{\ell=1}^{m}
\Hat{\bm{r}}_{\mathit{\ell}}^{{\prime}}({\Hat{\bm{R}}_{0}}^{-1}\otimes{\Hat{\bm{R}}_{0}}^{-1})\Hat{\bm{r}}_{\mathit{\ell}}/(n-\ell).
\end{equation}
In the univariate time series, the $\Tilde{Q}_m$ test statistic approximately equal
the Ljung-Box statistic \citep{Ljung1978}
and both statistics are asymptotically $\chi^2_{k^2(m-p-q)}$ \citep{Hosking1980B,Hosking1981B}.

\subsection{Univariate generalized variance portmanteau test}\label{Univariate}

\citet{PR2002} proposed a univariate portmanteau test statistic,
\begin{equation}\label{PenaRodriguezD}
\Hat{D}_m = n \left( 1-{\mid \mathcal{\Hat{R}}_{\mathit{m}}\mid}^{1/{\mathit{m}}} \right),
\end{equation}
where $\mid \bullet \mid$ denotes the determinant and
$\mathcal{\Hat{R}}_{\mathit{m}}$ is the residual correlation matrix of order $m+1$,
\begin{equation}\label{PenaRodrigueztoeplitz}
\mathcal{\hat{R}}_{\mathit{m}} = \left(
\begin{array}{cccc}
  1 & \hat{r}_{11}(1) & \ldots & \hat{r}_{11}(m) \\
  \hat{r}_{11}(1) & 1 & \ldots &\hat{r}_{11}(m-1) \\
  \vdots & \ldots & \ddots &  \vdots  \\
  \hat{r}_{11}(m)  & \hat{r}_{11}(m-1) & \dots & 1 \\
\end{array}
\right).
\end{equation}

\citet{PR2002} derived the asymptotic distribution of
$\Hat{D}_m$ as gamma using the standardized values of residual autocorrelations.
\citet[\S 2.7]{Li2004} noted several interesting interpretations for this statistic.
It was shown
in simulation experiments \citep{PR2002} that the $\Hat{D}_m$ statistic had better power than the
test of \citet{Ljung1978} in many situations.
One problem noted by \citet{Lin2006} is that the test statistic $\Hat{D}_m$
may not exist because, with the modified version of the
residual autocorrelations used, the residual autocorrelation sequence is not
always positive-definite or even non-negative definite.
Furthermore, the size of the test may not be accurate due to the asymptotic approximation
\citep[p. 19]{Li2004}.
To overcome these difficulties \citet{Lin2006} suggested using
a Monte-Carlo significance test
and demonstrated that this approach provides a test with
the correct size and is often more powerful than the usual Ljung-Box test
\citep[Table 6]{Lin2006}.

\citet{PR2006} suggested taking the log of the $(m+1)$th root of the
determinant in eqn.~(\ref{PenaRodrigueztoeplitz}),
\begin{equation}\label{PenaRodriguezTwo}
\Tilde{D}_m = -n(m+1)^{-1} \log{\mid \mathcal{\hat{R}}_{\mathit{m}}} \mid
\end{equation}
and they derived a gamma distribution approximation for this test statistic.

In the portmanteau tests based on the asymptotic distribution \citep{Ljung1978, PR2002, PR2006}
not only is the size of the test inaccurate if the series length $n$ is not large enough
but there is also a problem if $m$, the number of lags, is not large
enough as well.
The Monte-Carlo significance test approach does not require any such assumption about $m$ and
has much better finite-sample properties than
tests based on the asymptotic distribution.

\section{New Multivariate Portmanteau Test}\label{secNew}

The univariate residual autocorrelations in the Toeplitz matrix in
eqn. (\ref{PenaRodrigueztoeplitz})
are replaced by,
$\bm{\Hat{R}}_{\ell}, \ell= 1,\ldots, m$
in eqn. (\ref{HostingRACF}),
\begin{equation}\label{GV.mat1}
      \mathfrak{\Hat{\bm{R}}}_\mathit{m} = \left(%
\begin{array}{cccc}
  \mathbb{I}_k & \Hat{\bm{R}}_{1} & \ldots & \bm{\Hat{R}}_{\mathit{m}} \\
  \Hat{\bm{R}}^{\prime}_{1} & \mathbb{I}_k & \ldots &\Hat{\bm{R}}_{\mathit{m}-1} \\
  \vdots & \ldots & \ddots &  \vdots  \\
  \Hat{\bm{R}}^{\prime}_{\mathit{m}} & \Hat{\bm{R}}^{\prime}_{\mathit{m}-1}& \dots & \mathbb{I}_k \\
\end{array}%
\right),
\end{equation}
where $\mathbb{I}_k = \Hat{\bm{R}}_{0}$.
The proposed multivariate portmanteau test statistic is
\begin{equation}\label{NewD}
\mathfrak{D}_m = -n \log |\mathfrak{\Hat{\bm{R}}}_\mathit{m}|.
\end{equation}

From Hadamard's inequality for the determinant of a positive definite matrix,
$|\mathfrak{\Hat{\bm{R}}}_\mathit{m}|  \le 1$.
When there is no significant autocorrelation in the residuals,
$\Hat{\bm{R}}_{\ell} = O_p(n^{-\frac{1}{2}})$ so $\mathfrak{\Hat{\bm{R}}}_\mathit{m}$
is approximately block diagonal and hence
$|\mathfrak{\Hat{\bm{R}}}_\mathit{m}| \approx 1$.

On the other hand, when there is autocorrelation present, $|\mathfrak{\Hat{\bm{R}}}_\mathit{m}|$
will be expected to be smaller than $1$.
To see this we repeatedly apply the formula for the determinant of a partitioned matrix
\citep[\S 14.1]{Seber2008},
\begin{equation}
\label{NewDecomposition}
|\mathfrak{\Hat{\bm{R}}}_\mathit{m}|
  = \prod_{\ell=1}^{m}
 \mid \mathbb{I}_k -
 \mathfrak{\Hat{\bm{R}}}_{(\mathit{\ell})}
 \mathfrak{\Hat{\bm{R}}}^{-1}_{\mathit{\ell-1}}
 \mathfrak{\Hat{\bm{R}}}_{(\mathit{\ell})}^\prime \mid,
\end{equation}
where
$\mathfrak{\Hat{\bm{R}}}_{(\mathit{\ell})}
=   [\Hat{\bm{R}}_{1}:\cdots:\Hat{\bm{R}}_{\mathit{\ell}}]$
is the $k$-by-$\ell k$ block partitioned matrix.
Then
$\hat{\Sigma}_\ell = \mathbb{I}_k -
 \mathfrak{\Hat{\bm{R}}}_{(\mathit{\ell})}
 \mathfrak{\Hat{\bm{R}}}^{-1}_{\mathit{\ell-1}}
 \mathfrak{\Hat{\bm{R}}}_{(\mathit{\ell})}^\prime$
corresponds to the error covariance matrix when a linear predictor
of order $\ell$ is fit to  $\bm{\hat{L}}^{\prime} \Hat{\bm{a}}_{t}$
using the previous $\ell$ values \citep[eqn. (3.15)]{Reinsel1997}.
Thus, eqn. (\ref{NewDecomposition}) is a direct multivariate generalization of the well known
univariate decomposition of generalized variance into the product
of the one-step ahead variances of the linear minimum-mean-square error
predictors \citep[p. 532]{McLeod1977},
\begin{equation}\label{UnivariateMMSE}
\mid \mathcal{\Hat{R}}_{\mathit{m}}\mid =\prod_{\ell=1}^{m} \hat \sigma_\ell^2,
\end{equation}
where  $\Hat{\sigma}_\ell^2$ is the mean-square error for a fitted
linear predictor of order $\ell$.
In this case,  $R^2_\ell = 1-\hat \sigma_\ell^2$,
where $R_\ell^2$ is the square of the multiple correlation for
the order $\ell$ linear predictor,
and so \citep[eqn. (7)]{PR2002},
\begin{equation}\label{PRDecomposition}
\mid \mathcal{\Hat{R}}_{\mathit{m}}\mid =\prod_{\ell=1}^{m} (1-R_\ell^2).
\end{equation}
In the multivariate case,
\begin{equation}\label{MVR}
\hat{\eta}_\ell^2 = 1-
 \mid
 \mathbb{I}_k -
 \mathfrak{\Hat{\bm{R}}}_{(\mathit{\ell})}
 \mathfrak{\Hat{\bm{R}}}^{-1}_{\mathit{\ell-1}}
 \mathfrak{\Hat{\bm{R}}}_{(\mathit{\ell})}^{\prime}
 \mid
\end{equation}
is the proportion of the generalized variance that is accounted for by
a linear predictor of order $\ell$.
From eqns. (\ref{NewDecomposition}, \ref{MVR}),
the corresponding multivariate equivalent of eqn. (\ref{PRDecomposition}) is
\begin{equation}\label{NewDecompositionPartB}
|\mathfrak{\Hat{\bm{R}}}_\mathit{m}|
= \prod_{\ell=1}^{m} (1-\hat{\eta}_\ell^2).
\end{equation}
It follows from eqn. (\ref{NewDecompositionPartB}),
$|\mathfrak{\Hat{\bm{R}}}_\mathit{m}| < 1$
and that the smaller the value of
$|\mathfrak{\Hat{\bm{R}}}_\mathit{m}|$,
the more strongly autocorrelated the normalized
residuals,
$\bm{\hat{L}}^{\prime} \Hat{\bm{a}}_{t}$, are.

Using the \citet{Chitturi1974} multivariate residual autocorrelations, eqn. (\ref{ChitturiRACF}),
the correlation matrix corresponding to eqn. (\ref{GV.mat1}),
$\Hat{\mathfrak{\bm{R}}}^{(\ddag)}_\mathit{m}$,
is defined by the block matrix with $(i,j)$-block, $\Hat{\bm{R}}^{(\ddag)}_{i-j}$
for $i,j=1,\ldots,m+1$.
This matrix is not symmetric but
$ |\mathfrak{\Hat{\bm{R}}}_\mathit{m}|
=\mid \Hat{\mathfrak{\bm{R}}}^{(\ddag)}_\mathit{m}  \mid$,
so these multivariate autocorrelations could also be used.

Multivariate autocorrelations are often defined as in eqn. (\ref{LiMcLeodRACF})
\citep[eqn. (14.1.2)]{Box2008}.
Using this definition,
the residual autocorrelation matrix may be written,
\begin{equation}
\label{RLiMcLeod}
\Hat{\bm{R}}^{(\dag)}_{\ell} = \bm{\hat{D}}^{-1/2}\Hat{\bm{\Gamma}}_{\ell}\bm{\hat{D}}^{-1/2},
\end{equation}
where
$\bm{\hat{D}}^{-1/2} = \diag (\hat\gamma_{1,1}^{-1/2}(0),\ldots,\hat\gamma_{k,k}^{-1/2}(0))$.
The correlation matrix corresponding to eqn. (\ref{GV.mat1})
obtained by replacing $\Hat{\bm{R}}_{\ell}$ by $\Hat{\bm{R}}^{(\dag)}_{\ell}$
may be denoted by  ${\mathfrak{\Hat{\bm{R}}}}^{(\dag)}_\mathit{\ell}$
and the corresponding generalized variance portmanteau statistic,
$\mid {\mathfrak{\Hat{\bm{R}}}}^{(\dag)}_{\mathit{m}}\mid$.
A similar decomposition as given in eqn. (\ref{NewDecompositionPartB})
shows that small values $\mid {\mathfrak{\Hat{\bm{R}}}}^{(\dag)}_{\mathit{m}}\mid$
correspond to positive autocorrelation.
On the other hand, when there is no autocorrelation present,
the off-block diagonal entries in the matrix ${\mathfrak{\Hat{\bm{R}}}}^{(\dag)}_\mathit{m}$ are $O_p(n^{-1/2})$.
So, $ \mid {\mathfrak{\Hat{\bm{R}}}}^{(\dag)}_{\mathit{m}}\mid \approx  \mid\Hat{\bm{R}}^{(\dag)}_{0}\mid^{m+1}$.
When the innovation variance matrix, $\bm{\Gamma}_{0}$, has large off-diagonal elements,
$\mid\Hat{\bm{R}}^{(\dag)}_{0}\mid < 1$.
Hence again $\mid {\mathfrak{\Hat{\bm{R}}}}^{(\dag)}_{\mathit{m}}\mid = O_p(r^m)$ for some $r \in (0,1)$.
So, in both cases, autocorrelation or no autocorrelation,
$ \mid {\mathfrak{\Hat{\bm{R}}}}^{(\dag)}_{\mathit{m}}\mid$ tends to be small provided
the innovation covariance matrix is not diagonal.
Numerical experiments confirmed that
the test using $\mathfrak{D}^{(\dag)}_{m}$ and $\mathfrak{D}_{m}$ are essentially
equivalent when $\bm{\Gamma}_{0}$ is diagonal but in
the non-diagonal case, $\mathfrak{D}^{(\dag)}_{m}$ does not provide a useful test.

\subsection{Asymptotic distribution and approximation}\label{subAsyDis}

In this section, the asymptotic distribution for $\mathfrak{D}_m$
in eqn. (\ref{NewD}) is derived and an approximation to this
distribution is suggested.
Since, as shown in \citet[Figure 2]{Lin2006} in the univariate case by simulation,
the actual finite-sample distribution for $\mathfrak{D}_m$ converges slowly,
the asymptotic distribution for $\mathfrak{D}_m$  is not expected to be of much use in diagnostic
checking multivariate time series models unless $n$ is very large.

We use the following notation as in \citet[\S 4]{Hosking1980B},
$\bm{\Psi(B)}=\bm{\Phi(B)}^{-1}\bm{\Theta(B)}=\sum_{i=0}^{\infty}\Psi_iB^{i}$ and
$\bm{\Pi(B)}=\bm{\Theta(B)}^{-1}=\sum_{i=0}^{\infty}\Pi_i B^{i}$
are matrix power series such that the elements $\Psi_i$ and $\Pi_i$ converge
exponentially to zero as $i\rightarrow\infty$.
Define
\begin{equation}
\label{MatrixG}
\bm{G}= \left(%
\begin{array}{cccc}
  G_0 &0& \ldots &0 \\
  G_1 & G_0&\ldots&0 \\
  \vdots & \vdots& \ddots &  \vdots  \\
  G_{m-1} & G_{m-2}& \dots &G_{m-p} \\
\end{array}%
\right),
\end{equation}
and
\begin{equation}
\label{MatrixH}
\bm{H}= \left(%
\begin{array}{cccc}
  H_0 &0& \ldots &0 \\
  H_1 & H_0 &\ldots&0 \\
  \vdots & \vdots& \ddots &  \vdots  \\
  H_{m-1} & H_{m-2}& \dots &H_{m-q} \\
\end{array}%
\right),
\end{equation}
where
$G_r=\sum_{i=0}^{\infty}\bm{\Gamma}_{0}\Psi_i^{\prime}\otimes\Pi_{r-i}$
and
$H_r=\bm{\Gamma}_{0}\otimes\Pi_{r}$.

\theoremstyle{plain}\newtheorem{thm}{\smcaps{Theorem}}
\begin{thm}
Assume that the model specified in eqn. (\ref{VARMA})
has independent and identically distributed innovations with mean zero
and constant covariance matrix.
The model is fit to a series of length $n$ using
an $n^{-1/2}$-consistent algorithm.
After obtaining the residuals defined in eqn. (\ref{HostingRACF})
and the test statistic,  $\mathfrak{D}_m$, in eqn. (\ref{NewD}),
$\mathfrak{D}_m$ is asymptotically distributed as
$$\sum_{i=1}^{k^2m}\lambda_{i}\chi_{1,i}^2,$$
where
$\chi_{1,i}^2,~i=1,\ldots,k^2m$ are independent $\chi_{1}^2$ random variables
and $\lambda_{1},\ldots,\lambda_{k^2m}$
are the eigenvalues of $(\mathbb{I}_{k^2}-\bm{Q})\bm{M}$, where
$\bm{M}$ is $k^2m\times k^2m$ diagonal matrix
\begin{equation}\label{M}
\bm{M}= \left(%
\begin{array}{cccc}
  m\mathbb{I}_{k^2} & \bm{O} & \ldots &\bm{O} \\
  \bm{O} & (m-1)\mathbb{I}_{k^2} &\ldots&\bm{O} \\
  \vdots & \vdots & \ddots &  \vdots  \\
  \bm{O} & \bm{O}& \dots & \mathbb{I}_{k^2} \\
\end{array}%
\right),
\end{equation}
and
\begin{equation}\label{Q}
\bm{Q}=\bm{X(X^{\prime}W^{-1}X)^{-1}X^{\prime}W^{-1}}
\end{equation}
is an idempotent matrix with rank $k^2(p+q)$,
$\bm{X}$ is defined as $k^2m\times k^2(p+q)$ matrix $\bm{(G-H)}$,
and
$\bm{W}=\mathbb{I}_{m}\otimes\bm{\Gamma}_{0}\otimes\bm{\Gamma}_{0}$
is positive-definite symmetric.
\end{thm}

\begin{proof}[\smcaps{Proof}]
From the decomposition in eqn. (\ref{NewDecomposition}),
it follows that,
\begin{equation}
\label{DetMatrixHosking2}
-n\log |\mathfrak{\Hat{\bm{R}}}_\mathit{m}|
   = -n\sum_{\ell=1}^{m}
 \log\mid
 \mathbb{I}_k - \bm{A}_{\ell}
 \mid,
\end{equation}
where
$\bm{A}_{\ell}=
\mathfrak{\Hat{\bm{R}}}_{(\mathit{\ell})}
\mathfrak{\Hat{\bm{R}}}^{-1}_{\mathit{\ell-1}}
\mathfrak{\Hat{\bm{R}}}_{(\mathit{\ell})}^{\prime}$.
Using the fact that $\mid \mathbb{I}_k - \bm{A}_{\ell} \mid=\prod_{i=1}^{k}(1-\lambda_{i}(\ell))$,
where $\lambda_{i}(\ell)$ are the eigenvalues of $\bm{A}_{\ell}$, $\ell=1,\ldots,m$,
\begin{equation}
\label{DetMatrixHosking2eigen}
-n\log |\mathfrak{\Hat{\bm{R}}}_\mathit{m}|
   = -n\sum_{\ell=1}^{m}
 \sum_{i=1}^{k}\log(1-\lambda_{i}(\ell)).
\end{equation}
Expanding
$\log(1-\lambda_{i}(\ell))=-\sum_{r=1}^{\infty}r^{-1}{\lambda_{i}^{r}(\ell)}$
and
$\tr(\bm{A}_{\ell})=\sum_{i=1}^{k}\lambda_{i}(\ell)$,
\begin{equation}
\label{DetMatrixHosking3}
\mathfrak{D}_m
   = n\sum_{\ell=1}^{m}
    \tr(\bm{A}_{\ell})+O_{p}(n^{-1}).
\end{equation}
One can verify that
\begin{equation}
\label{DetMatrixHosking4}
\begin{aligned}
\tr(\bm{A}_{1})
 & =\tr(\hat{\bm{R}}_{1}^{\prime}\hat{\bm{R}}_{1})\\
\tr(\bm{A}_{2}) & \approx \tr(\hat{\bm{R}}_{1}^{\prime}\hat{\bm{R}}_{1})+\tr(\hat{\bm{R}}_{2}^{\prime}\hat{\bm{R}}_{2})\\
 & \vdots\\
\tr(\bm{A}_{m}) & \approx \tr(\hat{\bm{R}}_{1}^{\prime}\hat{\bm{R}}_{1})+\ldots+\tr(\hat{\bm{R}}_{m}^{\prime}\hat{\bm{R}}_{m}),
\end{aligned}
\end{equation}
so that,
\begin{equation}
\label{DetMatrixHosking5}
\mathfrak{D}_m
   \approx n\sum_{\ell=1}^{m}
    (m-\ell+1)\tr(\hat{\bm{R}}_{\ell}^{\prime}\hat{\bm{R}}_{\ell}).
\end{equation}

Using the commutative property of trace,
\begin{equation}
\label{DetMatrixHosking6}
\mathfrak{D}_m
   \approx n\sum_{\ell=1}^{m}
    (m-\ell+1)\tr({\Hat{\bm{\Gamma}}_{\mathit{\ell}}}^{{\prime}}{\Hat{\bm{\Gamma}}_{0}}^{-1}
    \Hat{\bm{\Gamma}}_{\mathit{\ell}}{\Hat{\bm{\Gamma}}_{0}}^{-1}).
\end{equation}
It follows from \citet[eq. (2.12)]{Neudecker1969},
\begin{equation}
\label{DetMatrixHosking7}
\begin{aligned}
\mathfrak{D}_m
   &\approx n\sum_{\ell=1}^{m}(m-\ell+1)
 (\VEC\Hat{\bm{\Gamma}}_{\ell})^{\prime}(\Hat{\bm{\Gamma}}_{0}^{-1}\otimes\Hat{\bm{\Gamma}}_{0}^{-1})\VEC\Hat{\bm{\Gamma}}_{\ell},\\
 &=n(\VEC\Hat{\bm{\Gamma}})^{\prime}(\mathbb{I}_{m}\otimes\Hat{\bm{\Gamma}}_{0}^{-1}\otimes\Hat{\bm{\Gamma}}_{0}^{-1})\bm{M}
 (\VEC\Hat{\bm{\Gamma}}),
\end{aligned}
\end{equation}
where
$\VEC\Hat{\bm{\Gamma}}=(\VEC\Hat{\bm{\Gamma}}_1\ldots \VEC\Hat{\bm{\Gamma}}_m)$ is $k^2m\times 1$ column vector
and
$\bm{M}$ is $k^2m\times k^2m$ diagonal matrix defined in eqn. (\ref{M}).

\citet[Theorem 1]{Hosking1980B} showed that
\begin{equation}
\label{HoskingVec}
\sqrt{n}\VEC\Hat{\bm{\Gamma}}\sim N_{k^2m}(\bm{0},(\mathbb{I}_{k^2m}-\bm{Q})\bm{W}),
\end{equation}
where
$\bm{W}^{-1}$ can be replaced by a consistent estimator
$\hat{\bm{W}}^{-1}=\mathbb{I}_{m}\otimes\Hat{\bm{\Gamma}}_{0}^{-1}\otimes\Hat{\bm{\Gamma}}_{0}^{-1}$,
and $\bm{Q}$ is the idempotent matrix with rank $k^2(p+q)$ in eqn. (\ref{Q}).

From the theorem on quadratic forms given by \citet[Theorem 2.1]{Box1954}, and
eqns. (\ref{DetMatrixHosking7}, \ref{HoskingVec}),
the asymptotic distribution of $\mathfrak{D}_m$ is given by,
\begin{equation}
\label{DetMatrixHosking8}
\mathfrak{D}_m\rightarrow\sum_{i=1}^{k^2m}\lambda_{i}\chi_{1}^2,
\end{equation}
where $\rightarrow$ stands for convergence in distribution as
$n \rightarrow \infty$ and
$\lambda_{1},\ldots,\lambda_{k^2m}$
are the eigenvalues of $(\mathbb{I}_{k^2m}-\bm{Q})\bm{M}$.
\end{proof}

\subsubsection{Approximation}\label{Approximation}

The upper percentiles of the cumulative distribution function
in eqn. (\ref{DetMatrixHosking8}) could be evaluated by the \citet{Imhof1961} algorithm.
For the univariate case, \citet[Table 2]{Lin2006} showed that the convergence
to the asymptotic distribution is very slow.
In the case of large-samples, an approximation based
on \citet[Theorem 3.1]{Box1954} works well.
Using this result, the test statistic in eqn. (\ref{DetMatrixHosking8})
can be approximated by $a\chi_{b}^2$,
where $a$ and $b$ are chosen to make the first two moments
agree with those of exact distribution
of $\mathfrak{D}_m$.
Hence,
$a=\sum\lambda_i^2/\sum\lambda_i$
and
$b=(\sum\lambda_i)^2/\sum\lambda_i^2$,
where,
\begin{equation}
\label{sumL}
\begin{aligned}
\sum_{i=1}^{k^2m}\lambda_{i} &= \tr(\mathbb{I}_{k^2m}-\bm{Q})\bm{M},\\\smallskip
\sum_{i=1}^{k^2m}\lambda_{i}^2 &=\tr(\mathbb{I}_{k^2m}-\bm{Q})\bm{M}(\mathbb{I}_{k^2m}-\bm{Q})\bm{M}.
\end{aligned}
\end{equation}

When $p=q=0$,
$a=(2m+1)/3$ and $b=1.5k^2m(m+1)/(2m+1)$.
In the \VARMA$(p,q)$ case,
one degree of freedom is lost for each parameter so
$\mathfrak{D}_m$ is approximately distributed as $a \chi_{b}^2$,
where
\begin{equation}
\label{chisquparameters}
\begin{aligned}
a&=\frac{2m+1}{3},\\\smallskip
b&=\frac{3k^2m(m+1)}{2(2m+1)}-k^2(p+q).
\end{aligned}
\end{equation}

\subsection{Monte-Carlo significance test}\label{subMonte}

Monte-Carlo significance tests, originally suggested by George Barnard \citep{Barnard},
are feasible for many small-sample problems \citep{Marriott}
and with modern computing facilities these types of tests are increasingly
feasible for larger samples and more complex problems \citep{Dufour01}.
For a pure significance test with no nuisance parameters, as is the case, for
example, for simply testing a time series for randomness,
accuracy of the Monte-Carlo procedure depends only on
the number of simulations \citep[Proposition 2.1]{Dufour2006}.

In the case of diagnostic checking, the model parameters must be estimated and
 \citet[Proposition 5.1]{Dufour2006} has shown that, provided consistent estimators are used,
Monte-Carlo tests remain asymptotically valid.
Since we assume $n^{-1/2}$-consistent estimators are used, the requirements
for  \citet[Proposition 5.1]{Dufour2006} are met.

Simulations for $\mathfrak{D}_m$ in the univariate case \citep[Table 3]{Lin2006}
as well as our simulations for the multivariate case in Section \ref{subSignificance},
suggest the impact of nuisance parameters
is negligible.
The p-value for all of the portmanteau test statistics presented in this
paper may be obtained using the Monte-Carlo method outlined below.
We use the statistic $\mathfrak{D}_m$ in the description but $\Tilde{Q}_m$
could be used instead.

\begin{description}
\item[Step 1:] Set $N$, the number of simulations. Usually, $N \leftarrow 1000$ but
smaller values may be used if necessary. By choosing $N$ large enough, an accurate
estimate of the p-value may be obtained.
\item[Step 2:] After fitting the model and obtaining the residuals, compute the portmanteau
test statistic for lag $m$ or possibly a set of lags such as $\ell=1, \ldots, m$, where $m \ge 1$.
Typically $m$ is chosen large enough to allow for possible high-order autocorrelations.
Denote the observed value of the test statistics by $\mathfrak{D}_\ell^{(o)}, \ell=1, \ldots, m$.
\item[Step 3:] For each $i=1,\ldots,N$, simulate the fitted model, refit it, obtain the
residuals from this model, compute the test statistic, $\mathfrak{D}_\ell^{(i)}, \ell=1, \ldots, m$.
\item[Step 4:] For each $\ell$, $\ell=1, \ldots, m$, the estimated p-value is given by,
\begin{equation}
\label{MCEquation}
       \hat{\mathsf{p}} =\frac{\#\{\mathfrak{D}_\ell^{(i)} \geq\mathfrak{D}_\ell^{(o)},~i=1,2,\ldots,N\}+1}{N+1}.
\end{equation}
\end{description}

The approximate 95\% margin of error for the p-value is,
$1.96 \sqrt{\hat{\mathsf{p}} (1- \hat{\mathsf{p}})/N}$.

The above algorithm is a simply a restatement of the Monte-Carlo testing algorithm
given by \citet[\S 3]{Lin2006} for the univariate case.
\citet[Table 3]{Lin2006} demonstrate that the Monte-Carlo testing
procedure has the correct size for an \AR$(1)$ and this is verified for some \VAR$(1)$
models in Section~\ref{subSignificance}.

\theoremstyle{remark}\newtheorem{rem}{\smcaps{Remark}}

\begin{rem}
In the Monte-Carlo test procedure it is assumed that the innovations
used in our simulations in Step 3 are normally distributed but
any distribution with constant covariance matrix could be used.
In particular, using the empirical joint distribution is equivalent to
bootstrapping the multivariate residuals.
Using bootstrapped residuals is implemented
in our software \citep{Mahdi2011}.
\end{rem}

\begin{rem}
A limitation of the Monte-Carlo diagnostic check
is the assumption of constant variance.
Many financial time  series exhibit conditional
heteroscedasticity.
In practice this means that our test may overstate
the significance level \citep{Duchesne2003}.
This means that when used for constructing a VAR or VARMA model,
the final fitted model may not be as parsimonious as a model
developed using a portmanteau test which takes into
conditional heteroscedasticity \citep{Francq2007, Duchesne2006}.
Our Monte-Carlo portmanteau test can also be used to test for the presence of multivariate conditional
heteroscedasticity simply by replacing the residuals by squared or absolute residuals.
An illustration of this procedure is given later in Section~\ref{subWestGerman}.
\end{rem}

\begin{rem}
\citet{Francq2007} discuss a more general asymptotic
multivariate portmanteau  diagnostic test that is valid
assuming only that the innovations are uncorrelated.
This test requires a large sample though.
\end{rem}

\begin{rem}
\citet{Lin2008} discuss the Monte-Carlo portmanteau test for
univariate ARMA with infinite variance.
The Monte-Carlo method of \citet{Lin2008} for infinite-variance ARMA has been extended
to the multivariate case as well and
is available in our R package \citep{Mahdi2011}.
\end{rem}

\section{Simulation results}\label{secSimulation}

The purpose of our simulations is to demonstrate the improved power
as well as the correct size of the
Monte-Carlo (MC) test using $\mathfrak{D}_m$.
We also compare the empirical Type 1 error rates for the $a \chi^2_b$-approximation
discussed in Section \ref{Approximation}.

\subsection{Comparison of type 1 error rates}\label{subSignificance}

The empirical error rates have been evaluated under the Gaussian bivariate \VAR(1) process
$\bm{Z}_{t}=\bm{\Phi}_i\bm{Z}_{t-1}+\bm{a}_{t},~i=1,\ldots,4$
for the portmanteau test statistic $\mathfrak{D}_m$ using the MC and $a \chi^2_b$-approximation
to evaluate the p-value.
The covariance matrix of $\bm{a}_{t}$ has unit variances and covariance $1/2$ and
the coefficient matrices are taken from \citet{Hosking1980B} and \citet{Li1981},
\renewcommand\minalignsep{10pt}
\begin{eqnarray*}
\bm{\Phi}_1=\left(%
\begin{array}{cc}
 0.9 & 0.1 \\
  -0.6 & 0.4 \\
\end{array}%
\right),~
\bm{\Phi}_2=\left(%
\begin{array}{cc}
  -1.5 & 1.2 \\
  -0.9 & 0.5 \\
\end{array}%
\right),
\bm{\Phi}_3=\left(%
\begin{array}{cc}
  0.4 & 0.1 \\
  -1.0 & 0.5 \\
\end{array}%
\right),~
\bm{\Phi}_4=\left(%
\begin{array}{cc}
  0.3 & 0.5 \\
  0.0 & 0.3 \\
\end{array}%
\right).
\label{PHIGAMMA}
\end{eqnarray*}

The empirical error rates are shown in Table 1.
For each entry in Table 1, $10^3$ simulations were done.
The MC test also used $N=10^3$.

The 95\% confidence interval assuming the a 5\% rejection rate for
each test is $(3.6, 6.4)$.
There are 17 entries outside this interval with the $a \chi^2_b$ approximation
only one 1 with the Monte-Carlo test.
In conclusion, size-distortion with the Monte-Carlo test appears
to be negligible but is sometimes present
when the $a \chi^2_b$ approximation is used.

In Section \ref{secApplication}, we found that there is
a much larger discrepancy between the p-values using
the $a \chi^2_b$ approximation and those using the Monte-Carlo test.

\begin{table}[htpb]
\begin{center}
\begin{tabular}{ccccccccccc}
\noalign{\smallskip}
\noalign{\hrule}
\noalign{\smallskip}
\noalign{\hrule}
\noalign{\smallskip}
&&\multicolumn{2}{c}{$n=100$}&&\multicolumn{2}{c}{$n=200$}&&\multicolumn{2}{c}{$n=500$}\\
\noalign{\smallskip}
\cline{3-4}\cline{6-7}\cline{9-10}
\noalign{\smallskip}
 & $m$ & $a \chi^2_b$ &MC  && $a \chi^2_b$ & MC && $a \chi^2_b$ & MC \\
\noalign{\smallskip}
\noalign{\hrule}
$\bm{\Phi}_1$ &    &   &   &&   &   &&   &    \\
 &   5 &  5.9 &  4.6 &&  5.1 &  4.7 &&  4.8 &  4.8  \\
 &  10 &  5.2 &  4.5 &&  4.4 &  5.2 &&  3.7 &  4.2  \\
 &  15 &  5.7 &  5.4 &&  4.5 &  4.4 &&  3.6 &  3.8  \\
 &  20 &  6.8 &  5.8 &&  4.8 &  4.0 &&  3.8 &  3.8  \\
 &  25 &  7.8 &  4.9 &&  5.3 &  4.1 &&  4.0 &  4.0  \\
 &  30 &  9.0 &  4.8 &&  5.8 &  3.7 &&  4.4 &  4.1  \\

$\bm{\Phi}_2$ &    &   &   &&   &   &&   &    \\
  &  5 &  4.7 &  4.8 &&  4.0 &  4.8 &&  3.5 &  4.7  \\
 &  10 &  4.8 &  3.8 &&  3.8 &  4.0 &&  3.5 &  4.8  \\
 &  15 &  5.7 &  3.9 &&  4.3 &  3.9 &&  3.6 &  5.0  \\
 &  20 &  6.9 &  4.2 &&  4.9 &  4.2 &&  3.8 &  4.8  \\
 &  25 &  8.2 &  4.0 &&  5.3 &  3.9 &&  4.1 &  5.3  \\
 &  30 &  9.5 &  4.3 &&  5.8 &  4.0 &&  4.5 &  5.4  \\

$\bm{\Phi}_3$ &    &   &   &&   &   &&   &    \\
  &  5 &  4.0 &  4.6 &&  3.6 &  5.7 &&  3.2 &  5.2  \\
 &  10 &  4.5 &  4.8 &&  3.8 &  6.5 &&  3.1 &  5.3  \\
 &  15 &  5.1 &  4.2 &&  4.1 &  6.3 &&  3.3 &  5.1  \\
 &  20 &  6.6 &  4.3 &&  4.6 &  6.2 &&  3.6 &  5.2  \\
 &  25 &  7.7 &  4.5 &&  5.3 &  5.4 &&  4.0 &  5.3  \\
 &  30 &  9.0 &  4.2 &&  5.9 &  5.5 &&  4.3 &  5.0  \\

$\bm{\Phi}_4$ &    &   &   &&   &   &&   &    \\
  &  5 &  2.9 &  4.3 &&  2.6 &  4.7 &&  2.5 &  5.2  \\
 &  10 &  3.9 &  4.6 &&  3.2 &  4.9 &&  3.0 &  4.5  \\
 &  15 &  4.9 &  4.1 &&  3.9 &  4.6 &&  3.2 &  5.0  \\
 &  20 &  6.1 &  4.4 &&  4.5 &  5.3 &&  3.6 &  4.9  \\
 &  25 &  7.3 &  3.9 &&  5.0 &  5.0 &&  3.9 &  4.8  \\
 &  30 &  8.7 &  3.9 &&  5.6 &  5.2 &&  4.3 &  4.7 \\
\noalign{\smallskip}
\noalign{\hrule}
\end{tabular}
\label{TableSigLevel}
\caption{The empirical $5\%$ significance level, in percent, comparing
approximation, $a \chi^2_b$, and Monte-Carlo, MC,
for the portmanteau test statistic $\mathfrak{D}_m$.
}
\end{center}
\end{table}

\goodbreak
\newpage

\subsection{Power comparisons}\label{subPower}

Only Monte-Carlo significance tests are used to
compare the empirical power of 5\% level
tests with $\Tilde{Q}_m$ and $\mathfrak{D}_m$.
Possible size-distortion sometimes makes power comparisons between asymptotic tests and
Monte-Carlo tests invalid.
In our comparisons, VAR models are fitted to various multivariate
models.
The power of diagnostic tests using $\mathfrak{D}_m$ versus $\Tilde{Q}_m$
are compared using simulation.
In all comparisons, the p-values were evaluated using
the Monte-Carlo (MC) method with $N = 10^3$.
We consider a \VAR(1) model fitted
to simulated data generated from eight \VARMA models selected
from well-known textbooks as cited below.

\subsubsection*{Model 1}

\citet[p. 17]{Lutkepohl2005}.

\begin{eqnarray*}\label{PowerEquation1}
    \left[%
    \begin{array}{c}
      Z_{1,t} \\
      Z_{2,t}
    \end{array}%
    \right]
    -\left[%
    \begin{array}{cc}
      0.5 & 0.1 \\
      0.4 & 0.5 \\
    \end{array}%
    \right]
    \left[%
    \begin{array}{c}
      Z_{1,t-1} \\
      Z_{2,t-1}
    \end{array}%
    \right]
    -\left[%
    \begin{array}{cc}
      0 & 0 \\
      0.3 & 0 \\
    \end{array}%
    \right]
    \left[%
    \begin{array}{c}
      Z_{1,t-2} \\
      Z_{2,t-2}
    \end{array}%
    \right]
    =\left[%
    \begin{array}{c}
      a_{1,t} \\
      a_{2,t}
    \end{array}%
    \right]
    \end{eqnarray*}
   \begin{equation*}
    \bm{\Gamma}_{0}=\left(%
    \begin{array}{cc}
     1.00 & 0.71 \\
     0.71 & 1.00 \\
    \end{array}%
    \right)
   \end{equation*}

\subsubsection*{Model 2}

\citet[p. 428]{Brockwell}.

    \begin{eqnarray*}\label{PowerEquation5}
      \left[%
      \begin{array}{c}
       Z_{1,t} \\
       Z_{2,t}
      \end{array}%
     \right]
     -\left[%
      \begin{array}{cc}
        0.7 & 0 \\
        0 & 0.6 \\
      \end{array}%
      \right]
      \left[%
      \begin{array}{c}
       Z_{1,t-1} \\
       Z_{2,t-1}
      \end{array}%
      \right]
      =\left[%
      \begin{array}{c}
        a_{1,t} \\
        a_{2,t}
      \end{array}%
      \right]-\left[%
      \begin{array}{cc}
       0.5 & 0.6 \\
      -0.7 & 0.8 \\
     \end{array}%
     \right]
     \left[%
     \begin{array}{c}
      a_{1,t-1} \\
      a_{2,t-1}
     \end{array}%
     \right]
     \end{eqnarray*}
\begin{equation*}
  \bm{\Gamma}_{0}=\left(%
  \begin{array}{cc}
  1.00 & 0.71 \\
  0.71 & 2.00 \\
  \end{array}%
  \right)
\end{equation*}

\subsubsection*{Model 3}

\citet[p. 81]{Reinsel1997}.

    \begin{eqnarray*}\label{PowerEquation6}
      \left[%
      \begin{array}{c}
       Z_{1,t} \\
       Z_{2,t}
      \end{array}%
     \right]
     -\left[%
      \begin{array}{cc}
        1.2 & -0.5 \\
        0.6 & 0.3 \\
      \end{array}%
      \right]
      \left[%
      \begin{array}{c}
       Z_{1,t-1} \\
       Z_{2,t-1}
      \end{array}%
      \right]
      =\left[%
      \begin{array}{c}
        a_{1,t} \\
        a_{2,t}
      \end{array}%
      \right]-\left[%
      \begin{array}{cc}
       -0.6 & 0.3 \\
       0.3 & 0.6 \\
     \end{array}%
     \right]
     \left[%
     \begin{array}{c}
      a_{1,t-1} \\
      a_{2,t-1}
     \end{array}%
     \right]
     \end{eqnarray*}
     \begin{equation*}
  \bm{\Gamma}_{0}=\left(%
  \begin{array}{cc}
  1.00 & 0.50 \\
  0.50 & 1.25 \\
  \end{array}%
  \right)
\end{equation*}

\subsubsection*{Model 4}

Tsay [2005 2nd ed, p. 371].

    \begin{eqnarray*}\label{PowerEquation7}
      \left[%
      \begin{array}{c}
       Z_{1,t} \\
       Z_{2,t}
      \end{array}%
     \right]
     -\left[%
      \begin{array}{cc}
        0.8 & -2 \\
        0 & 0 \\
      \end{array}%
      \right]
      \left[%
      \begin{array}{c}
       Z_{1,t-1} \\
       Z_{2,t-1}
      \end{array}%
      \right]
      =\left[%
      \begin{array}{c}
        a_{1,t} \\
        a_{2,t}
      \end{array}%
      \right]
      -\left[%
      \begin{array}{cc}
       -0.5 & 0 \\
      0 & 0 \\
     \end{array}%
     \right]
     \left[%
     \begin{array}{c}
      a_{1,t-1} \\
      a_{2,t-1}
     \end{array}%
     \right]
     \end{eqnarray*}
     \begin{equation*}
  \bm{\Gamma}_{0}=\left(%
  \begin{array}{cc}
  1.00 & 0.71 \\
  0.71 & 1.00 \\
  \end{array}%
  \right)
\end{equation*}

\subsubsection*{Model 5}

\citet[p. 25]{Reinsel1997}.

    \begin{eqnarray*}\label{PowerEquation9}
      \left[%
      \begin{array}{c}
       Z_{1,t} \\
       Z_{2,t}
      \end{array}%
     \right]
      =\left[%
      \begin{array}{c}
        a_{1,t} \\
        a_{2,t}
      \end{array}%
      \right]-\left[%
      \begin{array}{cc}
       0.8 & 0.7 \\
      -0.4 & 0.6 \\
     \end{array}%
     \right]
     \left[%
     \begin{array}{c}
      a_{1,t-1} \\
      a_{2,t-1}
     \end{array}%
     \right]
     \end{eqnarray*}
     \begin{equation*}
  \bm{\Gamma}_{0}=\left(%
  \begin{array}{cc}
  4 & 1 \\
  1 & 2 \\
  \end{array}%
  \right)
\end{equation*}

\subsubsection*{Model 6}

Tsay [2005 2nd ed, p. 350].

    \begin{eqnarray*}\label{PowerEquation10}
      \left[%
      \begin{array}{c}
       Z_{1,t} \\
       Z_{2,t}
      \end{array}%
     \right]
      =\left[%
      \begin{array}{c}
        a_{1,t} \\
        a_{2,t}
      \end{array}%
      \right]-\left[%
      \begin{array}{cc}
       0.2 & 0.3 \\
      -0.6 & 1.1 \\
     \end{array}%
     \right]
     \left[%
     \begin{array}{c}
      a_{1,t-1} \\
      a_{2,t-1}
     \end{array}%
     \right]
     \end{eqnarray*}
     \begin{equation*}
  \bm{\Gamma}_{0}=\left(%
  \begin{array}{cc}
  2 & 1 \\
  1 & 1 \\
  \end{array}%
  \right)
  \end{equation*}

\subsubsection*{Model 7}

\citet[p. 445]{Lutkepohl2005}.

   \begin{eqnarray*}\label{PowerEquation11}
    \left[%
    \begin{array}{c}
      Z_{1,t} \\
      Z_{2,t}
    \end{array}%
    \right]
    -\left[%
    \begin{array}{cc}
      0.5 & 0.1 \\
      0.4 & 0.5 \\
    \end{array}%
    \right]
    \left[%
    \begin{array}{c}
      Z_{1,t-1} \\
      Z_{2,t-1}
    \end{array}%
    \right]
    &-&\left[%
    \begin{array}{cc}
      0 & 0 \\
      0.25 & 0 \\
    \end{array}%
    \right]
    \left[%
    \begin{array}{c}
      Z_{1,t-2} \\
      Z_{2,t-2}
    \end{array}%
    \right]
    =\\
    &&\left[%
    \begin{array}{c}
      a_{1,t} \\
      a_{2,t}
    \end{array}%
    \right]
    -\left[%
      \begin{array}{cc}
       0.6 & 0.2 \\
       0 & 0.3 \\
     \end{array}%
     \right]
     \left[%
     \begin{array}{c}
      a_{1,t-1} \\
      a_{2,t-1}
     \end{array}%
     \right]
     \end{eqnarray*}
   \begin{equation*}
    \bm{\Gamma}_{0}=\left(%
    \begin{array}{cc}
     1.0 & 0.3 \\
     0.3 & 1.0 \\
    \end{array}%
    \right)
   \end{equation*}

\subsubsection*{Model 8}

\citet[p. 141]{Reinsel1992}.

   \begin{eqnarray*}\label{PowerEquation12}
    \left[%
    \begin{array}{c}
      Z_{1,t} \\
      Z_{2,t} \\
      Z_{3,t}
    \end{array}%
    \right]
    &-&\left[%
    \begin{array}{ccc}
      0.4 & 0.3 & -0.6 \\
      0.0 & 0.8 & 0.4 \\
      0.3 & 0.0 & 0.0 \\
    \end{array}%
    \right]
    \left[%
    \begin{array}{c}
      Z_{1,t-1} \\
      Z_{2,t-1} \\
      Z_{3,t-1} \\
    \end{array}%
    \right]
    =\\
    &&\left[%
    \begin{array}{c}
      a_{1,t} \\
      a_{2,t} \\
      a_{3,t}
    \end{array}%
    \right]
    -\left[%
      \begin{array}{ccc}
    0.7 & 0.0 & 0.0 \\
    0.1 & 0.2 & 0.0 \\
    -0.4 & 0.5 & -0.1 \\
     \end{array}%
     \right]
     \left[%
     \begin{array}{c}
      a_{1,t-1} \\
      a_{2,t-1} \\
      a_{3,t-1}
     \end{array}%
     \right]
     \end{eqnarray*}
   \begin{equation*}
    \bm{\Gamma}_{0}=\left(%
    \begin{array}{ccc}
     1.0 & 0.5 & 0.4\\
     0.5 & 1.0 & 0.7\\
     0.4 & 0.7 & 1.0\\
    \end{array}%
    \right)
   \end{equation*}

The power of the portmanteau statistics
$\mathfrak{D}_m$ and $\Tilde{Q}_m$ for nominal $5\%$ tests using the MC test
are shown in Table 2.
The power is evaluated for $10^4$ simulations for each parameter setting and
and $N = 10^3$ is used in the MC algorithm.
It is clear from Table 2 that the $\mathfrak{D}_m$ test
is often substantially more powerful than the $\Tilde{Q}_m$.
Only when $n=50$ and $m=30$ is the $\Tilde{Q}_m$ test more powerful
and this only occurs for Models 2 and 4.

\begin{table}[htpb]
\begin{center}
\begin{tabular}{ccccccccccc}
\noalign{\smallskip}
\noalign{\hrule}
\noalign{\smallskip}
\noalign{\hrule}
\noalign{\smallskip}
&&\multicolumn{2}{c}{$n=50$}&&\multicolumn{2}{c}{$n=100$}&&\multicolumn{2}{c}{$n=200$}\\
\noalign{\smallskip}
\cline{3-4}\cline{6-7}\cline{9-10}
\noalign{\smallskip}
Model & $m$ & $\mathfrak{D}_m$ &$\Tilde{Q}_m$  && $\mathfrak{D}_m$ & $\Tilde{Q}_m$ && $\mathfrak{D}_m$ & $\Tilde{Q}_m$ \\
\noalign{\smallskip}
\noalign{\hrule}
1 &   5 &  35 &  24 &&  68 &  53 &&  96 &  90  \\
1 &  10 &  24 &  16 &&  55 &  36 &&  90 &  73  \\
1 &  15 &  18 &  14 &&  46 &  30 &&  85 &  61  \\
1 &  20 &  13 &  13 &&  39 &  26 &&  80 &  52  \\
1 &  30 &  10 &  12 &&  30 &  23 &&  68 &  43  \\
2 &   5 &  70 &  48 && 100 &  94 && 100 & 100  \\
2 &  10 &  60 &  38 &&  99 &  82 && 100 & 100  \\
2 &  15 &  50 &  35 &&  99 &  75 && 100 & 100  \\
2 &  20 &  43 &  34 &&  97 &  70 && 100 &  99  \\
2 &  30 &  28 &  37 &&  93 &  64 && 100 &  97  \\
3 &   5 &  99 &  84 && 100 & 100 && 100 & 100  \\
3 &  10 &  96 &  64 && 100 &  99 && 100 & 100  \\
3 &  15 &  93 &  48 && 100 &  97 && 100 & 100  \\
3 &  20 &  88 &  39 && 100 &  91 && 100 & 100  \\
3 &  30 &  73 &  36 && 100 &  77 && 100 & 100  \\
4 &   5 &  51 &  27 &&  93 &  62 && 100 &  98  \\
4 &  10 &  37 &  24 &&  84 &  48 && 100 &  89  \\
4 &  15 &  27 &  22 &&  74 &  40 &&  99 &  81  \\
4 &  20 &  20 &  22 &&  65 &  37 &&  98 &  73  \\
4 &  30 &  13 &  22 &&  53 &  33 &&  95 &  65  \\
5 &   5 &  99 &  68 && 100 & 100 && 100 & 100  \\
5 &  10 &  95 &  46 && 100 &  93 && 100 & 100  \\
5 &  15 &  90 &  36 && 100 &  81 && 100 & 100  \\
5 &  20 &  83 &  32 && 100 &  72 && 100 & 100  \\
5 &  30 &  69 &  30 && 100 &  60 && 100 &  97  \\
6 &   5 &  83 &  45 && 100 &  90 && 100 & 100  \\
6 &  10 &  74 &  32 && 100 &  69 && 100 & 100  \\
6 &  15 &  62 &  28 &&  99 &  57 && 100 &  96  \\
6 &  20 &  54 &  28 &&  98 &  52 && 100 &  92  \\
6 &  30 &  40 &  27 &&  95 &  44 && 100 &  84  \\
7 &   5 &  29 &  21 &&  65 &  49 &&  97 &  91  \\
7 &  10 &  19 &  14 &&  53 &  33 &&  92 &  74  \\
7 &  15 &  14 &  12 &&  43 &  27 &&  86 &  61  \\
7 &  20 &  13 &  11 &&  35 &  22 &&  82 &  53  \\
7 &  30 &  11 &  11 &&  27 &  19 &&  72 &  41  \\
8 &   5 &  77 &  28 &&  96 &  85 && 100 & 100  \\
8 &  10 &  65 &  19 &&  92 &  61 && 100 &  99  \\
8 &  15 &  52 &  17 &&  84 &  48 && 100 &  94  \\
8 &  20 &  38 &  14 &&  76 &  40 && 100 &  90  \\
8 &  30 &  15 &  13 &&  55 &  33 && 100 &  78  \\
\noalign{\smallskip}
\noalign{\hrule}
\end{tabular}
\label{TablePower}
\caption{Empirical power comparison of $\mathfrak{D}_m$ and $\Tilde{Q}_m$
for a nominal $5\%$ test. Power is in percent.
$10^4$ simulations with $N=10^3$.}
\end{center}
\end{table}

\break\vfill\eject

\section{{Illustrative applications}}\label{secApplication}


\subsection{IBM and S\&P index}\label{subIBM}

\citet[Chapter 8]{Tsay2010} uses the portmanteau diagnostic test in
constructing a VAR model for the monthly log returns of
IBM stock and the S$\&$P 500 index for January $1926$ to December $2008$.
So here, $n=996$.
Univariate analysis for both of these series indicates the
presence of conditional heteroscedasticity \citep[p. 408]{Tsay2010}
but for forecasting purposes, we may consider a VAR model rather a more complex
VAR/GARCH model \citep{Weiss1984, Francq2007}.
There are $n=996$ and the AIC selects a VAR(5) model.
We found that the BIC selects a VAR(1) model.
Table 3 compares the p-values for the portmanteau
tests for the VAR($p$) for $p=1,3,5$.

These portmanteau tests suggest that the VAR(5)
is adequate and that the VAR(1) and VAR(3) both exhibit lack of
fit.
The VAR(4) is not shown but the results for this model are
similar to the VAR(3).
As noted in Remark 2,
the presence of conditional heteroscedasticity means that the
p-values in Table 3 are too small
and this implies that, possibly, a lower order model than
the VAR($5$) may be adequate.
This possibility could be investigated using the multivariate portmanteau
test of \citet{Francq2007}.

Table 3 also shows that $a \chi^2_b$ approximation for
the p-value of $\mathfrak{D}_m$ is inaccurate
whereas for $\Tilde{Q}_m$ the asymptotic approximation
agrees quite well with the Monte-Carlo result.

\begin{table}[ht]
\label{TableTsay}
\begin{center}
\begin{tabular}{lcccccccccccc}
  \hline
\noalign{\smallskip}
&  \multicolumn{4}{c}{VAR(1)}  & \multicolumn{4}{c}{VAR(3)}  & \multicolumn{4}{c}{VAR(5)} \\
\noalign{\smallskip}
&  \multicolumn{2}{c}{$a \chi^2_b$}  & \multicolumn{2}{c}{MC}
&  \multicolumn{2}{c}{$a \chi^2_b$}  & \multicolumn{2}{c}{MC}
&  \multicolumn{2}{c}{$a \chi^2_b$}  & \multicolumn{2}{c}{MC}   \\
\noalign{\smallskip}
m &$\mathfrak{D}_m$ & $\Tilde{Q}_m$ &$\mathfrak{D}_m$ & $\Tilde{Q}_m$ &$\mathfrak{D}_m$ & $\Tilde{Q}_m$ &$\mathfrak{D}_m$ & $\Tilde{Q}_m$ &$\mathfrak{D}_m$ & $\Tilde{Q}_m$ &$\mathfrak{D}_m$ & $\Tilde{Q}_m$ \\
  \hline
 5  & 0.2 & *   & * & *   & 10.4 & 0.6 &  1.7  & 0.6   & NA   &  NA   & 91.2  & 89.9 \\
 10 & 0.1 & 0.3 & * & 0.2 & 13.5 & 6.1 &  2.8  & 4.0   & 77.4 &  50.3 & 59.4  & 50.2 \\
 15 & 0.3 & 2.1 & * & 2.2 & 20.4 & 22.3&  6.4  & 22.1  & 84.0 &  61.2 & 63.1  & 61.4 \\
 20 & 0.2 & *   & * & *   & 15.4 & 2.6 &  5.0  & 2.2   & 71.8 &  11.3 & 45.2  & 9.9 \\
 25 & 0.1 & *   & * & *   & 8.7  & 1.1 &  2.3  & 0.7   & 53.0 &  7.6  & 27.5  & 7.1 \\
 30 & 0.2 & *   & * & *   & 7.3  & 2.7 &  2.3  & 2.2   & 46.2 &  13.7 & 23.3  & 12.0\\
   \hline
\end{tabular}
\caption{
IBM and S$\&$P 500 Index Data.
$a \chi^2_b$: approximation.
MC: Monte-Carlo $N=10^{3}$.
NA: not applicable.
The p-values are in percent.
The $*$ indicates a p-value less than 0.1$\%$.
}
\end{center}
\end{table}

\break\vfill\eject
\strut

\subsection{Investment, income and consumption time series}\label{subWestGerman}

The trivariate quarterly time series, 1960--1982, of West German investment, income, and consumption
was discussed by \citet[\S 3.2.3]{Lutkepohl2005}.
For this series, $n=92$ and $k=3$.
As in \citet[\S 4.3.1]{Lutkepohl2005} we model the logarithms of the first differences.
Using the AIC, \citet[Table 4.5]{Lutkepohl2005} selected a \VAR$(2)$
for this data.
Only lags $m=5,10,15$ are used in the diagnostic checks since $n$ is relatively short.
All diagnostic tests reject simple randomness, \VAR$(0)$.
The Monte-Carlo tests for \VAR$(1)$ suggests model inadequacy at lag 5.
Table 4 
supports the choice of the \VAR$(2)$ model.

\begin{table}[ht]
\begin{center}
\begin{tabular}{lccccccccccccc}
  \hline
\noalign{\smallskip}
&  \multicolumn{4}{c}{VAR(0)}  &  \multicolumn{4}{c}{VAR(1)}  & \multicolumn{4}{c}{VAR(2)} \\
\noalign{\smallskip}
&  \multicolumn{2}{c}{$a \chi^2_b$}  & \multicolumn{2}{c}{MC}
&  \multicolumn{2}{c}{$a \chi^2_b$}  & \multicolumn{2}{c}{MC}
&  \multicolumn{2}{c}{$a \chi^2_b$}  & \multicolumn{2}{c}{MC}   \\
\noalign{\smallskip}
m
&$\mathfrak{D}_m$ & $\Tilde{Q}_m$ &$\mathfrak{D}_m$ & $\Tilde{Q}_m$
&$\mathfrak{D}_m$ & $\Tilde{Q}_m$ &$\mathfrak{D}_m$ & $\Tilde{Q}_m$
&$\mathfrak{D}_m$ & $\Tilde{Q}_m$ &$\mathfrak{D}_m$ & $\Tilde{Q}_m$ \\
  \hline
 5 &  *  &  *  & 0.1 & 0.1 & 3.1 & 4.7  & 2.2  & 4.8  & 33.1 & 29.8 & 31.2 & 38.0 \\
10 &  *  & 0.6 & 0.3 & 0.5 & 4.0 & 14.7 & 7.0  & 12.7 & 49.5 & 48.0 & 54.2 & 50.6 \\
15 &  *  & 0.2 & 0.4 & 0.6 & 4.1 & 13.7 & 17.7 & 12.4 & 32.8 & 34.6 & 56.2 & 35.5 \\
   \hline
\end{tabular}
\label{TableWestGerman}
\caption{
Trivariate West German Macroeconomic Series.
$a \chi^2_b$: approximation.
MC: Monte-Carlo using $10^{3}$ replications.
The p-values are in percent and
$*$ indicates a p-value less than 0.1$\%$.
}
\end{center}
\end{table}

As pointed out in Remark 2, we may test for multivariate heteroscedasticity
by using the squared residuals and Table 5 gives the p-values with
this test for the VAR(2) model.
In this case, $a \chi^2_b$ approximation for $\mathfrak{D}_m$
as well as the asymptotic $\chi^2$ approximation for $\Tilde{Q}_m$
are quite inaccurate.
Based on the Monte-Carlo tests there is little evidence
to reject that null hypothesis of constant variance.

\begin{table}[ht]
\begin{center}
\begin{tabular}{lccccc}
  \hline
\noalign{\smallskip}
&  \multicolumn{5}{c}{VAR(2)} \\
\noalign{\smallskip}
&  \multicolumn{2}{c}{$a \chi^2_b$}  && \multicolumn{2}{c}{MC} \\
\noalign{\smallskip}
m
&$\mathfrak{D}_m$ & $\Tilde{Q}_m$ && $\mathfrak{D}_m$ & $\Tilde{Q}_m$\\
  \hline
 5 & 0.2  &  15.2 & & 31.9 & 81.3 \\
10 & 0.3  &  6.3  & & 24.4 & 37.9 \\
15 & *    &    *  & & 12.2 & 1.6 \\
   \hline
\end{tabular}
\label{TableHeteroscedastic}
\caption{The residuals of the fitted VAR(2) model on West German Macroeconomic series
are tested for heteroscedastic effects.
$a \chi^2_b$: approximation.
MC: Monte-Carlo using $10^{3}$ replications.
The p-values are in percent and
$*$ indicates a p-value less than $0.1\%$.
}
\end{center}
\end{table}

\newpage

\section{Concluding Remarks}\label{secConclusion}

\citet{Box2008} stress the importance of constructing
an adequate and parsimonious model in which the residuals pass a suitable portmanteau
diagnostic check.
In forecasting experiments with monthly riverflow time series,
\citet{Noakes1985} found that simply using a criterion
such as the AIC or BIC may provide a model that either
does not pass a suitable diagnostic check for randomness
of the residuals or that may have more parameters than necessary.
Monthly riverflow time series models chosen with the fewest number of parameters that
pass the portmanteau diagnostic check for periodic autocorrelation \citep{McLeod1994}
tend to produce better one-step ahead forecasts \citep{Noakes1985}.
\citet{McLeod1993} suggested formulating the principle of
parsimony as an optimization problem:  minimize
model complexity subject to model adequacy.
In any case, in the overall approach suggested many years ago
and presented in their recent book \citep{Box2008},
portmanteau diagnostic checks play a crucial role in constructing
time series models.

In Section~\ref{subMonte}, Remark 2, it was pointed out the Monte-Carlo test with $\mathfrak{D}_m$
may also be useful in diagnostic checking for multivariate conditional
heteroscedasticity when used with squared or absolute residuals.
This test is implemented in \citet{Mahdi2011}.
There is an extensive literature on
testing residuals in VAR and VARMA models for conditional heteroscedasticity
\citep{Ling1997, Duchesne2003, Duchesne2004, Rodriguez2005, Duchesne2006, Duchesne2008}.
The power study presented Section~\ref{subPower} suggests that
the $\mathfrak{D}_m$ with squared or absolute residuals may be useful.
\cite{PR2002}
also suggested that using squared-residuals with their generalized-variance
portmanteau test would outperform the usual diagnostic check \citep{McLeod1983}.
Other tests designed for particular alternatives might be
expected to perform better than an omnibus portmanteau test such as
$\mathfrak{D}_m$ or $\Tilde{Q}_m$ when these alternatives hold.
For example, \citet{Rodriguez2005} developed a diagnostic check for heteroscedasticity for the case of small
autocorrelations.

The multivariate portmanteau diagnostic test developed by \citet{Francq2007}
does not require independent and identically innovations but only
uncorrelated innovations.
This test would be appropriate for
the bivariate example in Section~\ref{subIBM}.

Scripts
for reproducing all tables in this paper are
available with our freely available software
\citep{Mahdi2011}.
This package can utilize multicore CPUs often found in modern personal computers
as well as a computer cluster or grid \citep{Schmidberger2009}.
On a modern eight core personal computer, the computations for Tables 4 and 5 take
about one minute.
Table 3 takes about six minutes due to the longer series
length and increased number of lags.
The simulations reported in Section \ref{secSimulation} were run on
a computer cluster.

\section*{Acknowledgements}
The authors would like to thank the referees for helpful comments that
greatly improved our paper.
This research was supported by an NSERC Discovery Grant awarded
to A.I. McLeod.

\bibliographystyle{plainnat}
\bibliography{JTSA3192R2}

\begin{thebibliography}{42}
\providecommand{\natexlab}[1]{#1}
\providecommand{\url}[1]{\texttt{#1}}
\expandafter\ifx\csname urlstyle\endcsname\relax
  \providecommand{\doi}[1]{doi: #1}\else
  \providecommand{\doi}{doi: \begingroup \urlstyle{rm}\Url}\fi

\bibitem[Barnard(1963)]{Barnard}
G.~A. Barnard.
\newblock Discussion of ``{T}he spectral analysis of point processes'' by {M}.
  {S}. {B}artlett.
\newblock \emph{Journal of the Royal Statistical Society, B}, 25:\penalty0
  264--296, 1963.

\bibitem[Box(1954)]{Box1954}
G.~E.~P. Box.
\newblock Some theorems on quadratic forms applied in the study of analysis of
  variance problems, {I}. {E}ffect of inequality of variance in the one-way
  classification.
\newblock \emph{The Annals of Mathematical Statistics}, 25\penalty0
  (2):\penalty0 290--302, 1954.

\bibitem[Box and Pierce(1970)]{Box1970}
G.~E.~P. Box and D.~A. Pierce.
\newblock Distribution of residual autocorrelation in autoregressive-integrated
  moving average time series models.
\newblock \emph{Journal of American Statistical Association}, 65\penalty0
  (332):\penalty0 1509--1526, 1970.

\bibitem[Box et~al.(2008)Box, Jenkins, and Reinsel]{Box2008}
G.~E.~P. Box, G.~M. Jenkins, and G.~C. Reinsel.
\newblock \emph{Time Series Analysis: Forecasting and Control}.
\newblock Wiley, New York, 4th edition, 2008.

\bibitem[Brockwell and Davis(1991)]{Brockwell}
P.~J. Brockwell and R.~A. Davis.
\newblock \emph{Time Series: Theory and Methods}.
\newblock Springer-Verlag, New York, 2nd edition, 1991.

\bibitem[{Chabot-Hallé} and Duchesne(2008)]{Duchesne2008}
D.~{Chabot-Hallé} and P.~Duchesne.
\newblock Diagnostic checking of multivariate nonlinear time series models with
  martingale difference errors.
\newblock \emph{Statistics and Probability Letters}, 78:\penalty0 997--1005,
  2008.

\bibitem[Chitturi(1974)]{Chitturi1974}
R.~V. Chitturi.
\newblock Distribution of residual autocorrelations in multiple autoregressive
  schemes.
\newblock \emph{Journal of the American Statistical Association}, 69\penalty0
  (348):\penalty0 928--934, 1974.

\bibitem[Duchesne(2004)]{Duchesne2004}
P.~Duchesne.
\newblock On robust testing for conditional heteroscedasticity in time series
  models.
\newblock \emph{Computational Statistics \& Data Analysis}, 46:\penalty0
  227--256, 2004.

\bibitem[Duchesne(2006)]{Duchesne2006}
P.~Duchesne.
\newblock Testing for multivariate autoregressive conditional
  heteroskedasticity using wavelets.
\newblock \emph{Computational Statistics \& Data Analysis}, 51\penalty0
  (353):\penalty0 2142--2163, 2006.

\bibitem[Duchesne and Lalancette(2003)]{Duchesne2003}
P.~Duchesne and S.~Lalancette.
\newblock On testing for multivariate {ARCH} effects in vector time series
  models.
\newblock \emph{The Canadian Journal of Statistics}, 31\penalty0 (3):\penalty0
  pp. 275--292, 2003.

\bibitem[Dufour(2006)]{Dufour2006}
{J.-M.} Dufour.
\newblock Monte {C}arlo tests with nuisance parameters: A general approach to
  finite-sample inference and non-standard asymptotics.
\newblock \emph{Journal of Econometrics}, 133\penalty0 (2):\penalty0 443--477,
  2006.

\bibitem[Dufour and Khalaf(2001)]{Dufour01}
{J.-M.} Dufour and L.~Khalaf.
\newblock Monte-{C}arlo test methods in econometrics.
\newblock In B.~Baltagi, editor, \emph{Companion to Theoretical Econometrics},
  chapter~23, pages 494--519. Blackwell, Oxford, 2001.

\bibitem[Francq and Ra\"isi(2007)]{Francq2007}
C.~Francq and H.~Ra\"isi.
\newblock Multivariate portmanteau test for autoregressive models with
  uncorrelated but nonindependent errors.
\newblock \emph{Journal of Time Series Analysis}, 28\penalty0 (3):\penalty0
  454--470, 2007.

\bibitem[Hosking(1980)]{Hosking1980B}
J.~R.~M. Hosking.
\newblock The multivariate portmanteau statistic.
\newblock \emph{Journal of American Statistical Association}, 75\penalty0
  (371):\penalty0 602--607, 1980.

\bibitem[Hosking(1981{\natexlab{a}})]{Hosking1981A}
J.~R.~M. Hosking.
\newblock Lagrange-multiplier tests of multivariate time-series models.
\newblock \emph{Journal of the Royal Statistical Society. Series B
  (Methodological)}, 43\penalty0 (2):\penalty0 219--230, 1981{\natexlab{a}}.

\bibitem[Hosking(1981{\natexlab{b}})]{Hosking1981B}
J.~R.~M. Hosking.
\newblock Equivalent forms of the multivariate portmanteau statistic.
\newblock \emph{Journal of the Royal Statistical Society, Series B},
  43\penalty0 (2):\penalty0 261--262, 1981{\natexlab{b}}.

\bibitem[Imhof(1961)]{Imhof1961}
J.~P. Imhof.
\newblock Computing the distribution of quadratic forms in normal variables.
\newblock \emph{Biometrika}, 48:\penalty0 419--426, 1961.

\bibitem[Li(2004)]{Li2004}
W.~K. Li.
\newblock \emph{Diagnostic Checks in Time Series}.
\newblock Chapman and Hall/CRC, New York, 2004.

\bibitem[Li and Mc{L}eod(1981)]{Li1981}
W.~K. Li and A.~I. Mc{L}eod.
\newblock Distribution of the residual autocorrelation in multivariate {ARMA}
  time series models.
\newblock \emph{Journal of the Royal Statistical Society, Series B},
  43\penalty0 (2):\penalty0 231--239, 1981.

\bibitem[Lin and Mc{L}eod(2006)]{Lin2006}
{J.-W.} Lin and A.~I. Mc{L}eod.
\newblock Improved {P}e$\tilde{\mbox{n}}$a-{R}odr\'{\i}guez portmanteau test.
\newblock \emph{Computational Statistics and Data Analysis}, 51\penalty0
  (3):\penalty0 1731--1738, 2006.

\bibitem[Lin and Mc{L}eod(2008)]{Lin2008}
{J.-W.} Lin and A.~I. Mc{L}eod.
\newblock Portmanteau tests for {ARMA} models with infinite variance.
\newblock \emph{Journal of Time Series Analysis}, 29\penalty0 (3):\penalty0
  600--617, 2008.

\bibitem[Ling and Li(1997)]{Ling1997}
S.~Ling and W.~K. Li.
\newblock Diagnostic checking of nonlinear multivariate time series with
  multivariate {ARCH} errors.
\newblock \emph{Journal of Time Series Analysis}, 18\penalty0 (5):\penalty0
  447--464, 1997.

\bibitem[Ljung and Box(1978)]{Ljung1978}
G.~M. Ljung and G.~E.~P. Box.
\newblock On a measure of lack of fit in time series models.
\newblock \emph{Biometrika}, 65:\penalty0 297--303, 1978.

\bibitem[L{\"u}tkepohl(2005)]{Lutkepohl2005}
H.~L{\"u}tkepohl.
\newblock \emph{New Introduction to Multiple Time Series Analysis}.
\newblock Springer-Verlag, New York, 2005.

\bibitem[Mahdi and Mc{L}eod(2011)]{Mahdi2011}
E.~Mahdi and A.~I. Mc{L}eod.
\newblock \emph{portes: Portmanteau Tests for ARMA, VARMA, ARCH, and FGN
  Models}, 2011.
\newblock URL \url{http://CRAN.R-project.org/package=portes}.

\bibitem[Marriott(1979)]{Marriott}
F.~H.~C. Marriott.
\newblock Barnard's {M}onte {C}arlo tests: How many simulations?
\newblock \emph{Applied Statistics}, 28\penalty0 (1):\penalty0 75--77, 1979.

\bibitem[Mc{L}eod(1977)]{McLeod1977}
A.~I. Mc{L}eod.
\newblock Improved {B}ox-{J}enkins estimators.
\newblock \emph{Biometrika}, 64\penalty0 (3):\penalty0 531--534, 1977.

\bibitem[McLeod(1993)]{McLeod1993}
A.~I. McLeod.
\newblock Parsimony, model adequacy and periodic correlation in forecasting
  time series.
\newblock \emph{International Statistical Review}, 61\penalty0 (3):\penalty0
  387--393, 1993.

\bibitem[Mc{L}eod(1994)]{McLeod1994}
A.~I. Mc{L}eod.
\newblock Diagnostic checking periodic autoregression models with application.
\newblock \emph{The Journal of Time Series Analysis}, 15:\penalty0 221--233,
  1994.
\newblock Addendum, JTSA 16, 647--648.

\bibitem[Mc{L}eod and Li(1983)]{McLeod1983}
A.~I. Mc{L}eod and W.~K. Li.
\newblock Diagnostic checking {ARMA} time series models using squared-residual
  autocorrelations.
\newblock \emph{Journal of Time Series Analysis}, 4:\penalty0 269--273, 1983.

\bibitem[Neudecker(1969)]{Neudecker1969}
H.~Neudecker.
\newblock Some theorems on matrix differentiation with special reference to
  {K}ronecker matrix products.
\newblock \emph{Journal of American Statistical Association}, 64\penalty0
  (327):\penalty0 953--962, 1969.

\bibitem[Noakes et~al.(1985)Noakes, Mc{L}eod, and Hipel]{Noakes1985}
D.~J. Noakes, A.~I. Mc{L}eod, and K.~W. Hipel.
\newblock Forecasting seasonal hydrological time series.
\newblock \emph{The International Journal of Forecasting}, 1\penalty0
  (1):\penalty0 179--190, 1985.

\bibitem[Pe\v{n}a and Rodr\'{\i}guez(2002)]{PR2002}
D.~Pe\v{n}a and J.~Rodr\'{\i}guez.
\newblock A powerful portmanteau test of lack of test for time series.
\newblock \emph{Journal of American Statistical Association}, 97\penalty0
  (458):\penalty0 601--610, 2002.

\bibitem[Pe\v{n}a and Rodr\'{\i}guez(2006)]{PR2006}
D.~Pe\v{n}a and J.~Rodr\'{\i}guez.
\newblock The log of the determinant of the autocorrelation matrix for testing
  goodness of fit in time series.
\newblock \emph{Journal of Statistical Planning and Inference}, 8\penalty0
  (136):\penalty0 2706--2718, 2006.

\bibitem[Poskitt and Tremayne(1982)]{Poskitt1982}
D.~S. Poskitt and A.~R. Tremayne.
\newblock Diagnostic tests for multiple time series models.
\newblock \emph{Annals of Statistics}, 6\penalty0 (1):\penalty0 114--120, 1982.

\bibitem[Reinsel(1997)]{Reinsel1997}
G.~C. Reinsel.
\newblock \emph{Elements of Multivariate Time Series Analysis}.
\newblock Springer-Verlag, New York, 2nd edition, 1997.

\bibitem[Reinsel et~al.(1992)Reinsel, Basu, and Yap]{Reinsel1992}
G.~C. Reinsel, S.~Basu, and S.~F. Yap.
\newblock Maximum likelihood estimators in the multivariate autoregressive
  moving average model from a generalized least squares viewpoint.
\newblock \emph{Journal of Time Series Analysis}, 13\penalty0 (2):\penalty0
  133--145, 1992.

\bibitem[Rodr\'{\i}guez and Ruiz(2005)]{Rodriguez2005}
J.~Rodr\'{\i}guez and E.~Ruiz.
\newblock A powerful test for conditional heteroscedasticity for financial time
  series with highly persistent volatilities.
\newblock \emph{Statistica Sinica}, 15:\penalty0 505--525, 2005.

\bibitem[Schmidberger et~al.(2009)Schmidberger, Morgan, Eddelbuettel, Yu,
  Tierney, and Mansmann]{Schmidberger2009}
Markus Schmidberger, Martin Morgan, Dirk Eddelbuettel, Hao Yu, Luke Tierney,
  and Ulrich Mansmann.
\newblock State of the art in parallel computing with \proglang{R}.
\newblock \emph{Journal of Statistical Software}, 31\penalty0 (1):\penalty0
  1--27, 2009.
\newblock URL \url{http://www.jstatsoft.org/v31/i01}.

\bibitem[Seber(2008)]{Seber2008}
G.~A.~F. Seber.
\newblock \emph{A Matrix Handbook for Statisticians}.
\newblock Wiley, New York, 2008.

\bibitem[Tsay(2010)]{Tsay2010}
R.~S. Tsay.
\newblock \emph{Analysis of Financial Time Series}.
\newblock Wiley, New York, 3rd edition, 2010.
\newblock 2nd edition, 2005.

\bibitem[Weiss(1984)]{Weiss1984}
A.~A. Weiss.
\newblock {ARMA} models with {ARCH} errors.
\newblock \emph{Journal of Time Series Analysis}, 5\penalty0 (2):\penalty0
  129--143, 1984.

\end{thebibliography}
\end{document}